\documentclass[12pt]{article}

\usepackage{amsmath,amssymb,amsthm,epsfig,graphics,amsfonts}
\def\cal{\mathcal}

\def\ord{{\textup{ord}}}
\numberwithin{equation}{section}

\title{Reversible Biholomorphic Germs}

\author{Patrick Ahern\\
Mathematics Department,
University of Wisconsin\\
Madison, WI 53706,
USA\\
e-mail: ahern@math.wisc.edu\\
 and\\
 Anthony G. O'Farrell$^1$
\\
Mathematics Department,
National University of Ireland\\
Maynooth, Co. Kildare,
Ireland\\
e-mail: anthonyg.ofarrell@gmail.com
}
\begin{document}



\font\ef=eufm10
\def\fraktur#1{{\ef {#1}}}

\def\bb#1{{\mathbb{#1}}}
\def\cal#1{{\mathcal{#1}}}
\def\fk#1{{\hbox{\fraktur#1}}}

\newtheorem{theorem}{Theorem}[section]%
\newtheorem{axiom}{Axiom}[section]%
\newtheorem{lemma}[theorem]{Lemma}%
\newtheorem{corollary}[theorem]{Corollary}%
\newtheorem{proposition}[theorem]{Proposition}%
\newtheorem{definitions}[theorem]{Definitions}
\newtheorem{problem}{Problem}[section]

\def\Proof{{\par\medskip\noindent\bf Proof. }}
\def\qed{{\hfill\vrule height 4pt width 4pt depth 0pt
             \par\vskip\baselineskip}}
\def\Wlg{{Without loss in generality, we may assume that }}
\def\wlg{{without loss in generality }}
\def\Example{\subsection{Example}}
\def\Examples{\subsection{Examples}} 
\def\Exercise{\subsection{Exercise}} 
\def\Exercises{\subsection{Exercises}} 
\def\Remark{\subsection{Remark}} 
\def\Remarks{\subsection{Remarks}} 
\def\Notation{\subsection{Notation}} 

\def\implies{{\ \Rightarrow\ }}
\def\mapsonto{{\rightarrow\hbox{\hskip -9pt \hbox{$\rightarrow$}}}}

\def\B{{\bb B}}
\def\C{{\bb C}}
\def\R{{\bb R}}
\def\N{{\bb N}}
\def\Q{{\bb Q}}
\def\R{{\bb R}}
\def\Rc{{\R^c}}
\def\Rd{{\R^d}}
\def\S{{\bb S}}
\def\U{{\bb U}}
\def\Z{{\bb Z}}

\def\mod{{\,\hbox{mod}\,}}
\def\Oh{{\hbox{O}}}
\def\oh{{\hbox{o}}}
\def\coeff{{\rm coeff}}
\def\dist{{\hbox{dist}}}
\def\dim{{\rm dim}}
\def\order{{\rm order}}
\def\diam{{\rm diam}}
\def\clos{{\rm clos}}
\def\bdy{{\rm bdy}}
\def\side{{\rm side}}
\def\span{{\rm span}}
\def\spt{{\rm spt}}
\def\dom{{\rm dom}}
\def\im{{\rm im}}
\def\diag{{\rm diag}}
\def\ball{{\rm ball}}
\def\Tan{{\rm Tan}}

\def\half{{\raise1pt\hbox{$\scriptscriptstyle{1\over 2}$}}}
\def\third{{\scriptstyle{1\over3}}}
\def\fifth{{\scriptstyle{1\over5}}}
\def\twothirds{{\scriptstyle{2\over3}}}
\def\quarter{{\scriptstyle{1\over4}}}

\def\deg{{^\circ}}	
\def\ONE{1\hskip-6pt1}

\maketitle

\footnotetext[1]{Supported by Grant SFI RFP05$/$MAT0003 and the 
ESF Network HCAA.}
\footnotetext[2]{Mathematics Subject Classification 2000: 30D05, 39B32, 37F99, 30C35.}

\section*{Abstract}
Let $G$ be a group. We say that an element $f\in G$
is {\em reversible in} $G$ if it is conjugate to its inverse,
i.e. there exists $g\in G$ such that $g^{-1}fg=f^{-1}$.
We denote the set of reversible elements by $R(G)$.
For $f\in G$, we denote by $R_f(G)$ the set (possibly empty)
of {\em reversers} 
of $f$, i.e. the set of $g\in G$ such that
$g^{-1}fg=f^{-1}$.  We characterise the elements of
$R(G)$ and describe each $R_f(G)$, where $G$
is the the group of biholomorphic germs in one complex variable.
That is, we determine all solutions to the equation
$ f\circ g\circ f = g$, in which $f$ and $g$ are holomorphic
functions on some neighbourhood of the origin, with
$f(0)=g(0)=0$ and $f'(0)\not=0\not=g'(0)$.

\section{Introduction}

\subsection{General Setting}
Let $G$ be a group. We say that an element $f\in G$
is {\em reversible in} $G$ if it is conjugate to its inverse,
i.e. there exists $g\in G$ such that $g^{-1}fg=f^{-1}$.
We denote the set of reversible elements by $R(G)$.
For $f\in G$, we denote by $R_f(G)$ the set (possibly empty)
of {\em reversers} 
of $f$, i.e. the set of $g\in G$ such that
$g^{-1}fg=f^{-1}$. 

The set $R(G)$ always includes the set $I(G)$
of involutions (elements of order at most 2). Indeed, it also
includes the larger set 
$$I^2(G) = \{\tau_1\tau_2: \tau_i\in I(G)\}$$
of {\em strongly-reversible elements}, i.e. elements that
are reversed by an involution.

If $g\in G$ reverses $f\in G$, then $g^2$
commutes with $f$, i.e. $g^2$ belongs to the
centraliser $C_f(G)$. More
generally, the composition of any two elements
of $R_f(G)$ belongs to $C_f(G)$. For this reason,
an understanding of centralisers in $G$ is a prerequisite
for an understanding of reversers.

The following easily-proved theorem characterises the reversers
of an element, in any group.

\begin{theorem}[Basic Theorem]\label{theorem-basic}
Let $G$ be a group and $f,g\in G$. Then the following three
conditions are equivalent:
\begin{enumerate}
\item $g\in R_f(G)$;
\item there exists $h\in G$ with $g^2=h^2$ and
$f=g^{-1}h$;
\item there exist $h\in G$ such that $f=gh$
and $f^{-1}=hg$.
\end{enumerate}
\end{theorem}
\qed
This then yields two characterisations of reversibility:

\begin{corollary}
Let $G$ be a group and $f\in G$. Then the following three
conditions are equivalent:
\begin{enumerate}
\item $f\in R(G)$;
\item there exist $g,h\in G$ with $g^2=h^2$ and
$f=g^{-1}h$;
\item there exist $g,h\in G$ such that $f=gh$
and $f^{-1}=hg$.
\end{enumerate}
\end{corollary}
\qed

This shows that reversibility is interesting only
in nonabelian groups in which there are elements with
multiple square roots.  In any specific group, it is
interesting to give more explicit characterisations
of reversibility than those of this theorem.

\medskip
This paper is about the reversible elements in the group
of invertible biholomorphic germs and some of its subgroups.

We shall characterise these elements, and their reversers,
and the strongly-reversible elements, in explicit ways.
We shall also consider some related questions.

The theory of reversibility for formal power series in one variable
has already been dealt with in \cite{OF}.
\medskip
We shall see (cf. Section \ref{section-reversers})
that there exist 
germs $f\in G$ that are
formally reversible, but not holomorphically reversible.

\subsection{Our specific groups}
For the remainder of the paper, we shall denote
by $G$ the group of biholomorphic germs at $0$ 
in one complex variable. Thus an element
of $G$ is represented by some function $f$,
holomorphic on some neighbourhood (depending on $f$)
of $0$, with $f'(0)\not=0$, and two such functions represent
the same germ if they agree on some neighbourhood
of $0$.  The group operation is composition. 
The identity is the germ of the identity function
$\ONE$. 

The {\em multiplier map} $m:f\to m(f)=f'(0)$
is a homomorphism from $G$ onto the multiplicative
group $\C^\times$ of the complex field.

Obviously, since $\C^\times$ is abelian,
the value $m(f)$
depends only on the conjugacy class of $f$ in $G$.

We denote
$$\begin{array}{rcl}
H&=& \{f\in G: m(f) = \exp(i\pi q), \textup{ for some }
q\in\Q\},\\
H_0 &=& \{f\in G: m(f)=\pm1\}= \ker m^2,\\
\textup{and   }\hskip1cm&&\\
G_1 &=& \ker m.
\end{array}
$$
These normal subgroups have
$ G_1 \le H_0 \le H \le G$.

Further, for $p\in\N$, we define 
$$ G_p=\{f\in G_1:
f^{(k)}(0)=0 \textup{ whenever }2\le k\le p\},$$
and 
$$ A_p = G_p\sim G_{p+1}.$$
Then $G_1$ is the disjoint union of $\{\ONE\}$ and the sets $A_p$.
For $f\in G_1$, with $f\not=\ONE$,
 we denote by $p(f)$ the unique
$p$ such that $ f\in A_p$.
The natural number $p(f)$ is a conjugacy invariant
of $f$ (with respect to conjugation in $G$), so that
each $G_p$ is a normal subgroup of $G$.

For $f\in G_p$, we may write
$f(z)=z+f_{p+1}z^{p+1}+O(z^{p+2})$. The map
$f\mapsto f_{p+1}$ is a group homomorphism from $G_p$
onto $(\C,+)$. Thus $f_{p+1}$ is a conjugacy
invariant of $f$ in $G_p$. It is even invariant under conjugation
in $G_1$, but it is not invariant under conjugation in $G$.
Each $f\in A_p$ may be conjugated to the form
$g^{-1}fg= z+z^{p+1}+a(f)z^{2p+1}+O(z^{2p+2})$, and then the
complex number $a(f)$
is a conjugacy invariant of $f$ in $G$.

The invariants $p(f)$ and $a(f)$ classify the elements
of $G_1\sim\{\ONE\}$ up to formal conjugacy. The complete 
biholomorphic conjugacy classification requires additional
invariants, and these have been provided by the equivalence class
of the EV data 
$\Phi(f)$ of  
\'Ecalle-Voronin theory, which is reviewed briefly in Section
\ref{section-EV} below. 

For $f\in H_0\sim G_1$, a complete set of conjugacy invariants 
(with respect to conjugacy in $G$) is
provided by $m(f)=-1$ and the conjugacy class
of $f^2$, which belongs to $G_1$. (See Theorem \ref{theorem-powers} below.)

\subsection{Summary of results}
It is obvious that each group homomorphism maps 
the reversible elements of its
domain to reversible
elements of its target, and that the only reversible elements
in an abelian group are its involutions.
Hence $R(G)\subset H_0$.  Consequently, the reversible elements
in all subgroups of $G$ lie in $H_0$.

Also, it is always true that for $f\in G_1$,
$p(f)=p(f^{-1})$.
Also, by purely formal considerations 
\cite{OF}, the condition $a(f)=a(f^{-1})$ is equivalent to
$a(f)=(p(f)+1)/2$.
Thus the short answer to the question of which $f\in G$
are reversible in $G$ is the following:

\begin{proposition}\label{proposition-EV}
Let $f\in G$. Then $f\in R(G)$ if and only if (exactly)
one of the following holds:
\begin{enumerate}
\item $f'(0)=1$, and $\Phi(f)$
is equivalent to $\Phi(f^{-1})$;
\item $f'(0)=-1$, and $f^2\in R(G)$.
\end{enumerate}
\end{proposition}

For Part 2, see 
Corollary \ref{corollary-square}.
However, 
we can provide much more explicit
information about reversibility in $G$.

In general groups, a reversible element $f$ may have
no reversers of finite order.  If there is a reverser
of finite order, then there is one whose order is a 
positive power of $2$. Only involutions can have a reverser of odd order.
In our present group $G$, we have the following:

\begin{theorem}\label{theorem-reversers}
Let $f\in A_p$, for some $p\in \N$, and $g\in R_f(G)$. 
Then $g$ has finite even order $2s$, for some $s\in \N$
with $p/s$ an odd integer.
\end{theorem}

We shall give examples (cf. Section \ref{section-reversers})
to show that there are $f\in G$
for which the lowest order of a reverser is any preassigned
power of $2$. 

We can be rather more precise about the order of reversers,
but we have to distinguish between \lq\lq flowable"
and \lq\lq non-flowable" reversible germs $f$. 

{\bf Definition}.
By a {\em flow} in $G_1$ we mean a continuous
group homomorphism $t\mapsto f_t$ from 
$(\R,+)$ (a {\em real flow}) or $(\C,+)$ (a {\em complex flow})
into $G_1$.  

 A germ $f\in G_1$ is called {\em flowable}
if and only if there exists a flow $(f_t)$ with
$f_1=f$.

The more precise result about reversers
involves technical parameters that are associated to
a reversible germ $f\in G_1$, and we shall give the statement and proof
later (cf. Section \ref{section-last}),
after we have explained these parameters.

\begin{theorem}\label{theorem-reversible}
let $f\in A_p$, for some $p\in \N$. Then $f\in R(G)$
if and only if it may be written as $g^{-1}h$, where
$g,h\in H$ are germs of finite even order $2s$,
$g^2=h^2$, $s|p$, and $p/s$ is odd. 
\end{theorem}

As is well-known, each germ of finite order in $G$
is conjugate in $H$ to a rotation through a rational multiple
of $\pi$ radians. 
Indeed 
an elements $g \in G$ of finite order $\delta$
must have multiplier $\beta=m(g)$ a $\delta$-th root
of unity, and is conjugate in
$H$ to $z\mapsto \beta z$; in fact the function
$$ \frac1{\delta}\left({z+ \frac{g(z)}{\beta} + 
\cdots + \frac{g^{\delta-1}(z)}{\beta^{\delta-1}}}\right)$$
provides a conjugation.

\begin{theorem}\label{theorem-series}
Let $f\in A_p$, for some $p\in \N$. Then $f\in R(G)$
if and only if there exists $\psi\in H$ such that
\begin{equation}\label{equation-series-1}
(\psi^{-1}f\psi)(z) = 
z + z^{p+1} + \sum_{k=1}^{\infty}c_kz^{sk+p+1},
\end{equation}
where $p/s$ is an odd integer, and
\begin{equation}\label{equation-series-2}
(\psi^{-1}f^{-1}\psi)(z) = 
z - z^{p+1} + \sum_{k=1}^{\infty}(-1)^kc_kz^{sk+p+1}.
\end{equation}
(In other words, $
f_1=\psi^{-1}f\psi$  is reversed by $z\mapsto \exp(\pi i /s)z$.)
\end{theorem}

We shall give examples (cf. Section \ref{section-reversers})
to show that each $p\in\N$ and each $s|p$
with $p/s$ odd may occur.

These results allow us to understand reversibility in $G$: One
reverses a germ $f$ essentially by \lq\lq rotating" it (using a rotation
modulo conjugacy), so as to swap the attracting and repelling petals
of its Leau flower.

We note some consequences:

\begin{corollary}\label{corollary-reversible-square}
Let $f\in G$. Then $f\in R(G)$ if and only if $f^2\in R(G)$.
\end{corollary}

\medskip
The strongly-reversible elements of $G$ were already
identified (in terms of EV data) 
in \cite{AG}, but we note the result, which follows 
immediately from Theorem \ref{theorem-reversers} above:

\begin{corollary}\label{corollary-strongly-reversible} 
Let $f\in G$. Then $f\in I^2(G)$
if and only if $f\in R(G)$ and one of the following holds:
\begin{enumerate}
\item $f\in I(G)$, or
\item $f\in A_p$ with $p$ odd.
\end{enumerate}
\end{corollary}

We note that the case $p=1$ was already given by Voronin \cite{V2}.

The following summarises our conclusions about
reversibility in all the above-named subgroups of $G$:

\begin{corollary} For each $p\in\N$, we have
$$ (\ONE)=R(G_p) = R(G_1)\subset 
R(H_0)=I^2(G)\subset R(H)=R(G)\subset H_0,$$
and the three inclusions are proper.
\end{corollary}

\section{Conjugacy}\label{section-EV}

{\bf Definition.}
Let $p\in\N$. Let $\fk S$ denote the set of all functions
$h$ that are defined and holomorphic on some
upper half-plane (depending on $h$), and are such that
$h(\zeta)-\zeta$ is bounded and has period 1.
By 
{\it \'Ecalle-Voronin $p$-data} (or just EV data) 
we mean an ordered $2p$-tuple $\Phi = (\Phi_1,\ldots,\Phi_{2p})$,
where $\Phi_1(\zeta)$,$-\Phi_2(-\zeta)$,$\Phi_3(\zeta)$,
$\ldots$,$-\Phi_{2p}(-\zeta)\in\fk S$.  

Given EV $p$-data $\Phi$ and $q$-data $\Psi$, we say that they are
{\em equivalent} if $p=q$ and there exist $k\in\Z$ and complex constants  
$c_1$,$\ldots$,$c_{2p}$, such that for each $j$ we have
$$\Phi_{j+2k}(\zeta+c_j) = \Psi_j(\zeta)+c_{j+1},$$
(where we define $\Phi_j$, $\Psi_j$ and $c_j$ for all $j\in\Z$
by making them periodic in $j$, with period $2p$).

\medskip
Let $f\in G_1$. Let $p=p(f)$. 
Voronin \cite{V1} described how to associate 
\'Ecalle-Voronin data $\Phi(f) = (\Phi_1,\ldots,\Phi_{2p})$
to $f$. We shall not recapitulate the construction here\footnote{
For a detailed description, see Voronin's paper 
\cite{V1} or (for full details when $p>1$)
\cite[pp.7-19]{AG}. The case $p>1$ was first fully elaborated by
Yu. S. Ilyashenko \cite{Ilya}.},
but roughly speaking the $\Phi_j$ are obtained as (analytic extensions
of) compositions
$F_j\circ F_{j+1}^{-1}$, where the $F_j$ are conformal maps of
alternately attracting and repelling Leau petals for $f$,
which conjugate $f$ on the petals to translation by $1$ near $\infty$.
Essentially the same construction  was discovered independently
by \'Ecalle \cite{M}. They proved the following:

\begin{theorem}[Conjugacy] Let $f,g\in G_1$. Then
$f$ is conjugate to $g$ in $G$ if and only if
$\Phi(f)$
is equivalent to $\Phi(g)$. 
\end{theorem}

\begin{theorem}[Realization]
Given any EV data $\Phi$, there exists a function
$f\in G_1$ having equivalent EV data. 
\end{theorem}

For $f\in H$, the expositions in print
usually say that the conjugacy classification 
is easily reduced to the case of multiplier $1$.
We need to consider multiplier $-1$, 
so we need a precise statement.
The result goes back to Muckenhoupt
\cite[Theorem 8.7.6, p. 359]{KCG}. 

\begin{theorem}[Muckenhoupt]\label{theorem-powers}
Suppose that $f,g\in H$ both have the same multipier $\lambda$,
a primitive $s$-th root of unity, where $s\in \N$.
Then $f$ and $g$ are conjugate in $G$ if and
only if $f^s$ and $g^s$ are
conjugate in $G$.
\end{theorem}

We supply a proof, partly for the reader's convenience, but
also because we wish to draw a useful corollary
from it.

\Proof It is evident that if $h^{-1}fh=g$, then
$h^{-1}f^sh= g^s$. 

For the other direction, 
suppose that there exists $h\in G$ with
$h^{-1}f^sh= g^s$. 

We have 
$(h^{-1}fh)^s= g^s$, and $m(h^{-1}fh)=\lambda$. 
So it suffices to show that 
$$ \left\{
\begin{array}{rcl}
m(f)&=&m(g)=\lambda \and 
\\
f^s&=&g^s
\end{array}
\right\} \implies f \hbox{ is conjugate to }g.$$ 
Let $k=f^s$. Then $k\in G_1$.

If $k$ is the identity, then $f$ and $g$ are periodic with
the same multiplier, so they are conjugate.

If $k$ is not the identity, then the centraliser of $k$
is abelian (see Theorems \ref{theorem-3.1} and \ref{theorem-3.2} 
below). Since $f$ and $g$ belong to it, 
they commute with each other, hence
$(f^{-1}g)^s=f^{-s}g^s=\ONE$. But $f^{-1}g\in G_1$,
so $f^{-1}g=\ONE$, and $f$ is actually equal
to $g$.
\qed

\begin{corollary}\label{corollary-powers}
If $f,g\in H$ have as multiplier the same 
$n$-th root of unity,
and $f^n\not=\ONE$,
then each $h\in G$ that conjugates  $f^n$ to $g^n$
will also conjugate $f$ to $g$.
\end{corollary}

\section{Centralisers}
The facts about $C_f(G)$, for $f\in G_1$,
were established by Baker and Liverpool \cite{Baker1,Baker2,Baker3,Liverpool}
 (see also
Szekeres \cite{S}).

We may summarise the facts about centralisers as follows:

\begin{theorem}\label{theorem-3.1}
Suppose that $p\in\N$ and $f\in A_p$ is flowable.
Then $C_f(G)$ is an abelian group, equal to the 
inner direct product 
$$\{ f_t : t\in\C\}\times\{ \omega^j: 
0\le j\le p\}$$
where $(f_t)_{t\in \C}$ is a complex flow, and
$\omega\in H$ has finite order $p$.
\end{theorem}

It follows from Theorem \ref{theorem-3.1}
that if $f\in G_1$ is flowable then $C_f(G_1)$ is the flow $(f_t)_{t\in \C}$.  
It is a remarkable result of Baker and Liverpool that in 
the non-flowable case $C_f(G_1)$ is an abelian group with a 
single generator $g$.  
Since $f\in C_f(G_1)$ we have $f=g^d$ for some integer $d$ 
(which we can assume to be positive (by replacing 
$g$ by $g^{-1}$ if necessary).  This $g$, which is unique, 
is usually denoted by $f^{\frac1d}$.

\begin{theorem}\label{theorem-3.2}
Suppose $f\in A_p$ is not flowable. Then $C_f(G)$
is abelian, and there exist positive 
integers $q$ and $\delta$ with $\delta|q$ and $q|p$ and elements 
$\tau$ and $\omega \in C_f(G)$ such that 
\begin{enumerate}
\item $C_f(G)/C_f(G_1)$ is cyclic of order $q$, 
\item $C_f(G)$ is generated 
by $\tau$ and $f^{1/d}$ 
\item $\omega$ has finite order $\delta$, 
\item we have a direct product decomposition 
$C_f(G)=\langle \tau\rangle \times\langle \omega\rangle $, and 
finally 
\item 
we have the relation 
$$\tau^{\frac{q}{\delta}}=\omega f^{1/d}.$$
\end{enumerate}
\end{theorem}

The formal centraliser of an $f\in G_1$ (other than $\ONE$)
is always isomorphic to the product of a flow and a finite cyclic 
group. Thus $C_f(G_1)$ is isomorphic to an
additive subgroup of $\C$.  The achievement of Baker and Liverpool
was to show that the only possible subgroups
that can occur are $\C$ itself and an infinite cyclic
group $\Z\alpha$, for some $\alpha\in\C$.  In the latter
case, $f$ has only a finite number of compositional roots. 
In particular, if $f$ is real-flowable, or  
infinitely-divisible, or lies in the image of a
$\Z^2$ action, then
it must be complex-flowable.

Voronin \cite{V1} used the EV data to characterise divisibility
of the elements $f\in G_1$, i.e. the existence
of composition roots.  In fact, for a given $f\in G_1$ and
$k\in\N$, there exists $g\in G_1$ with $g^k=f$, 
if and only if $\Phi=\Phi(f)$ satisfies
$$ \Phi_j(\zeta+\frac1k)=\Phi_j(\zeta)+\frac1k,$$
for $j=1,\ldots,2p(f)$.

In view of the Realisation Theorem, this means that
generic $f\in G_1$ have no roots at all.

The above theorems are deep, but may be proved rather more easily
than in the the original papers, by using Voronin's
approach \cite{V1}.
The flowable $f\in G_1$ are characterised as those
that have EV data equivalent to  
$\Phi_j(\zeta)=\zeta+\lambda_j$, for constant
$\lambda_j$, i.e. data that are translations.

\section{Reversers}\label{section-reversers}
After these preliminaries, we are ready to discuss
reversibility in $G$. First, we deal with
the case $m(f)=-1$. Then we proceed to prove the results
stated in Section 1.3, and to provide the examples promised.

\subsection{Multiplier $-1$}
First, we deal with the case $m(f)=-1$. From Corollary
\ref{corollary-powers} we deduce:

\begin{corollary}\label{corollary-square}
Let $f\in G$ have $f'(0)=-1$. Then 
(i) $f$ is an involution or $R_f(G) = R_{f^2}(G)$, and (ii) $f\in R(G)$ $\Leftrightarrow$
$f^2\in R(G)$. 
\end{corollary}

\subsection{Proof of Theorem \ref{theorem-reversers}}
We make use of formal series arguments below. It is 
also possible to prove some of the results by considering
separately the flowable and non-flowable germs, and using the
Baker-Liverpool theory on the latter.

Let $\fk G$ denote the group of formally-invertible
series, under the operation of formal composition.

To prove Theorem \ref{theorem-reversers}, fix $p\in\N$, a reversible
$f\in A_p$,
and $g\in R_f(G)$. 

Since $f\in R(G)$, then considered as a formal series, it belongs
to $R(\fk G)$. Hence \cite[Corollary 6]{OF} there exists a formal
series $\tau\in R_f(\fk G)$, of order $2p$.

Formally, $f$ is uniquely flowable \cite{Baker1}, i.e. there exists
a unique flow $(f^t)_{t\in\C}$ in $\fk G$ with $f^1=f$. Also,
$C_f(\fk G)$ is the set generated by $\tau^2$ and the $f^t$, $t\in\C$.
 This is well-known
\cite{Baker1,Liverpool,Lubin}, but quite concretely
$f$ is formally-conjugate \cite[Theorem 5]{OF}  to
$$ \frac{z}{(1+ z^p)^{1/p}},$$
and the same conjugacy
takes $f^t(z)$ to 
$$ \frac{z}{(1+t z^p)^{1/p}}.$$
For all $t\in\C$, the latter commutes with 
$z\mapsto \exp(2\pi i/p)z$,
and
is reversed $z\mapsto \exp(\pi i/p)z$,
and $\tau$ is obtained by conjugating the latter back. 

In particular, $\tau$ reverses each $f^t$, for $t\in\C$.

Now $\tau^{-1}g\in C_f(\fk G)$, and hence $\tau^{-1}g=\tau^{2r}f^t$
for some $r\in\Z$ and $t\in\C$, so
$g=\tau^mf^t$ for 
for an odd $m\in\Z$.
Since $\tau^m$ reverses $f^t$, we get
$g^2=\tau^mf^tf^{-t}\tau^m=\tau^{2m}$,
so the order of $g^2$ divides $p$, so the order
of $g$ is finite, dividing $2p$.

The order of $g$ cannot be odd (since $f$ is not involutive),
and hence it is $2s$, for some $s|p$.  Finally, if $p/s$ 
were even, we would have $m(g)^p=1$, but a simple formal
calculation shows that $g$ cannot reverse $f$ unless
$m(g)^p=-1$. \qed

\subsection{Proof of Theorem \ref{theorem-reversible}}

This is immediate from Corollary 1.2(2) and Theorem \ref{theorem-reversers}.

\subsection{Proof of Theorem \ref{theorem-series}}

Suppose $f\in R(G)$. By Theorem \ref{theorem-reversers}, there
exists $g\in R_f$, of order $2s$, with $p/s$ odd.  Thus
there is a function $\psi\in H$ that conjugates $g$ to
$\beta z$, where $\beta=m(g)$.

Then $\psi^{-1}f\psi$ is reversed by $\beta z$, and commutes
with $\beta^2z$. Since $\beta^2$ is a primitive
$s$-th root of unity, it follows that $\psi^{-1}f\psi$
takes the form given by equation (\ref{equation-series-1}).
Since $\beta z$ reverses it, 
$$ \psi^{-1}f^{-1}\psi(z) = \beta^{-1}
(\psi^{-1}f\psi)(\beta z)$$
takes the form (\ref{equation-series-2}).

This proves one direction, and the converse is
obvious. \qed

\subsection{Proof of Corollary \ref{corollary-reversible-square}}

It is true in any group that $f\in R(G)\implies f^2\in R(G)$. 
For the converse in our specific $G$,
there are two cases: $m(f)=\pm1$.

If $m(f)=1$, and $f^2\in R(G)$, then we have seen
in the proof of Theorem \ref{theorem-reversers} that 
each reverser of $f^2$ reverses each element of the
formal flow $(f^2)^t$, and hence reverses
$(f^2)^{1/2}=f$.  (Observe that if a convergent series
is formally reversed by a convergent series, then
it is holomorphically reversed by it, too.)

If $m(f)=-1$, and $f^2\in R(G)$, then we have $f\in R(G)$ by 
Proposition \ref{proposition-EV}, Part 2.

\subsection{Example: Reversible germ, not reversible by
any germ of order dividing $2^k$}
Fix any even $p\in\N$, and take $s=p$.
Let $\mu\in G$ be multiplication by a primitive
$s$-th root of $-1$. Take $\phi\in G_1$ commuting with 
$\mu^2$, but not with $\mu$. (This may be done, for instance,
by taking $\phi(z)=z+z^{s+1}$.)  Take
$g=\mu$, $h=\phi^{-1}\mu\phi$, and $f=g^{-1}h$. Then
a calculation shows that $g^2=h^2$ has order
$s$ (and hence $g$ is a reverser for $f$
of order $2s$),  and that $f\in A_p$. 
In case $p=2^{k+1}$, 
we see 
(by Theorem \ref{theorem-reversers}) that no element
of order $2^k$ can reverse $f$. 

\medskip
Another example is provided by the function
$z(1+z^p)^{-1/p}$ used in the proof of Theorem 1.4, 
in view of Corllary 1.8.
Examples of this kind may also be constructed (rather
less concretely) by appealing to the Realization Theorem).
However, the Realization Theorem is the best way to do 
the next thing:

\subsection{Example: Non-flowable reversible germ}
Fix any $p\in\N$, and take EV data $\Phi$, where
$$\Phi_1(\zeta) = \zeta + \exp(-2\pi i\zeta),
\qquad\Phi_2(\zeta) = \zeta - \exp(2\pi i\zeta),
$$
and $\Phi_{j+2}=\Phi_j$ for all $j$.

By the Realization Theorem, there is some $f\in A_p$
with EV data $\Phi(f)$ equivalent to $\Phi$. Hence
this $f$ is reversible, by Proposition \ref{proposition-EV},
because $(-\Phi_{j+1}(-\zeta))$ is the EV data for $f^{-1}$. 
(This is so, because the consecutive attracting and repelling
petals for $f$ are, respectively, repelling and attracting for
$f^{-1}$, and because $F_{j+1}$ conjugates $f^{-1}$ in the $j+1$-st
petal to
to $\zeta\mapsto\zeta-1$ near $\infty$, so that
$-F_{j+1}(-\cdot)$ conjugates $f^{-1}$ to
$\zeta\mapsto\zeta+1$, so that the EV recipe gives
$ -F_{j+1}(--F^{-1}_{j+2}(-\zeta)) = -\Phi_{j+1}(-\zeta) $
as EV data for $f^{-1}$.)

But since $\Phi_1$ is not a translation, $f$ is not 
flowable.

\subsection{Example: Formally-reversible germ, not reversible
in $G$}

Let ${\Phi}_1(\zeta)=\zeta +e^{-2\pi
i\zeta}$ and ${\Phi}_2(\zeta)=\zeta$.  If $f$ realizes this EV data then
$a(f)=1=(p+1)/2$ by the formula on top of page 19 of \cite{AG},
and hence f is formally reversible,
but these data do not have the symmetry required
of reversible germ data.

\section{The Order of a Reverser}\label{section-last}

Flowable reversible germs $f\in A_p$ are very special: they form a single conjugacy class -- all are conjugate to
$ z/(1+z^p)^{1/p}$, and all reversers for them
have order dividing $2p$. The possible orders are precisely
the divisors of $2p$ of the form $2^ku$, where
$u|p$ is odd, and $2^k$ is the largest power
of $2$ dividing $2p$.

\medskip
In the nonflowable case, we can relate the possible orders
for reversers to the centraliser generators $\tau$, $\omega$, and the
natural numbers $d$, $q$ and $\delta$ of Theorem
\ref{theorem-3.2}.  The numbers $d$, $q$, and $\delta$ are
uniquely-determined by $f$: the $1/d$-th power of $f$ is
the smallest positive power that converges, $q$ is the index of $C_f(H_1)$
in $C_f(G)=C_f(H)$, and $\delta$ is the order of the 
(cyclic) torsion subgroup of $C_f(G)$. The germ $\omega$
may be any generator of this torsion subgroup; we may 
specify a unique $\omega$ by requiring that the multiplier 
$m(\omega)
=e^{\frac{2\pi i}{\delta}}$ 
(as opposed to some other
primitive $\delta$-th root of unity).   

\begin{theorem} Let $p\in\N$, and suppose $f\in A_p$ is 
reversible but not flowable.  Let  
$\tau,\omega$ and $d,q,\delta$ be as in Theorem \ref{theorem-3.2}.
Then \begin{enumerate}
\item\label{last-1}  If $g\in R_f(G)$ then $g$ commutes with $\omega$,
and $g$ reverses $f^{r/d}$, for each $r\in\Z$. 
\item\label{last-2}  
  $\delta=q$, and $\frac{p}{q}$ is odd. 
\item\label{last-3}  
If we choose $\omega$ 
such that $m(\omega) =e^{\frac{2\pi i}{\delta}}$,
then we have 
$$\{g^2:g\in R_f(G)\}=\{\omega^l:l \hbox{ is odd}\},$$ 
and we  always have
 $$\{\ord(g):g\in R_f\}=\{2r\in\N: r|
q, \hbox{ and } q/r \hbox{ is odd}\}.$$
\end{enumerate}
\end{theorem}
\begin{proof}
We abbreviate $R_f=R_f(G)$.
(\ref{last-1})
Since $g$ (and hence $g^{-1}$) reverse $f$ and $\omega$ commutes with $f$ we 
see that $g\omega g^{-1}$ commutes with $f$,  
has order $\delta$ and has the same 
multiplier as $\omega$, and so it equals $\omega$.  
To show the second part of \ref{last-1}, it suffices to deal with the case 
$r=1$.  Again $gf^{\frac1d}g^{-1}$ commutes with $f$ and it has 
multiplier $1$ so $gf^{\frac1d}g^{-1}=f^{\frac{l}d}$ for some $l$.  Raise 
both sides of the last equation to the power d to get $f^{-1}=f^l$ and so 
$l=-1$ as desired. This proves part \ref{last-1}.

\medskip\noindent
(\ref{last-2})
We know that if $g\in R_f$ then $g'(0)^p=-1$, $g^2$ commutes with $f$ and 
that $g$ has finite order.  It follows that $g'(0)=e^{\frac{\pi im}{p}}$ where 
$m$ is odd.  Since $g^2$ is periodic and commutes with $f$ we have 
$g'(0)^{2\delta}=1$ i.e. $e^{2\pi im
\delta\over p}=1$.  This means that $m={\frac{p}{\delta}}l$ 
for some integer $l$.  Since $m$ is odd, so also are 
$\frac p{\delta}$ and $l$.  So far we have seen that 
$\frac{p}{\delta}$ is odd.  Now we show 
that $q=\delta$.  Now $g\tau g^{-1}
{\tau}^{-1}$ commutes with $f$ and has multiplier 
$1$ so $g\tau g^{-1}
{\tau}^{-1}=f^{\frac nd}$ for some integer $n$.  If we take 
this last identity and raise both sides to the power q we get 
$gf^{\delta\over 
d}g^{-1}f^{-\delta\over d}=f^{qn\over d}$.  Now using the fact that g reverses 
$f^{l\over d}$ we arrive at $-2\delta=qn$.  
So $-2={q\over \delta}n$ so that 
${q\over \delta}$ is either 1 or 2.  
But $q=2\delta$ is not consistent with the fact that
${p\over \delta}$ is odd.  
Hence $q=\delta$.

\medskip\noindent
(\ref{last-3})
Pick any $g\in R_f$.  We already know from Theorem \ref{theorem-reversers}
that $g$ has finite order. Since $g^2\in C_f$, it follows that
$g^2$ belongs to the torsion subgroup, and hence
is a power $\omega^l$. If $l$ were even, then $m(g)^p=1$,
but a reverser of $f$ must have $m(g)^p=-1$.  This proves that
$$\{g^2:g\in R_f\}\subset \{\omega^l:l \hbox{ is odd}\}.$$ 
To see the opposite inclusion, fix $g_0\in R_f$,
with $g_0^2=\omega^l$. Then
$\omega^j g_0\in R_f$ whenever  $j\in\Z$, and the square of this
reverser is $\omega^{2j}g_0^2=\omega^{l+2j}$.  Letting $j$ run through
$\delta$ consecutive integers, we get each odd power of $\omega$. 
Thus 
$$\{g^2:g\in R_f\}= \{\omega^l:l \hbox{ is odd}\}.$$ 

\medskip
We conclude that the possible values of $\ord(g)$
are the numbers $2\,\ord(\omega^l)$, where $l$
ranges over the odd numbers. Since $\omega$
has order $\delta=q$, the order of $\omega^l$
is $r=q/u$, where
$u$ is the greatest common divisor of $l$ and $\delta$. 
Since $l$ is odd, $u$ must be odd as well. 
Conversely,
suppose that $r$ is a divisor of $q$ and
$u=q/r$ is odd. 
Then by the last equation there is a $g\in R_f$ with 
$g^2={\omega}^u$, which obviously has order $r$.
\end{proof}

\begin{corollary}  If $p=2^ku$ where $u$ is odd, and
$f\in A_p$ is nonflowable and reversible in $G$, 
then $\delta=q=2^kn$ where n divides 
$u$.  The largest order for a reverser of f is $2\delta$ and the smallest 
order is $2^{k+1}$.
\end{corollary}\
\qed

Note that in the flowable case, this corollary also holds
(with, additionally, $q=p$).

\medskip\noindent
Using EV 
theory it can be shown that given any positive integer p and any divisor 
$q$ of $p$ such that ${p\over q}$ is odd then there is  a reversible 
$f\in A_p$ such that the associated ${q}_f=q$, and in fact 
an infinite dimensional set of 
inequivalent ones.

\end{document}